\documentclass[12pt]{article}
\usepackage{tipa}
\usepackage{bbm}
\usepackage{mathrsfs}
\usepackage{amsfonts}
\usepackage{cite}
\usepackage{amssymb,amsmath,amsthm}
\allowdisplaybreaks

\def\cl{\centerline}

\numberwithin{equation}{section}
\theoremstyle{plain}
\newtheorem{theo}{Theorem}[section]
\newtheorem{rem}[theo]{Remark}
\newtheorem{lemm}[theo]{Lemma}

\theoremstyle{definition}

\newtheorem{defi}[theo]{Definition}

\begin{document}
\cl{\large\bf{ Structure of a class of Lie conformal algebras of
Block type}\footnote{Supported by a grant from China Scholarship
Council (201708645021),  National Natural Science
Foundation grants of China (11661063, 11401570), and the Fundamental Research Funds for the Central Universities (2019QNA34).\\
\indent\ Corresponding author(W. Wang): wwll@mail.ustc.edu.cn}}
\vspace{6pt}

\centerline{Wei Wang$\,^{\dag }$, Chunguang Xia$\,^{\ddag}$,
 Li Liu$\,^{\dag}$}  \vspace{5pt}

\centerline{\small $^{\dag}$School of Mathematics and Computer
Science, Ningxia University, Yinchuan 750021,  China }

\centerline{\small $^{\ddag}$School of Mathematics, China University
of Mining and Technology, Xuzhou 221116, China } \vspace{6pt}

%\centerline{\small Email: wwll@mail.ustc.edu.cn,
%chgxia@cumt.edu.cn}

 {\small
\parskip .005 truein
\baselineskip 3pt \lineskip 3pt

\noindent{{\bf Abstract:} Let $p$ be a nonzero complex number.
Recently, a class of infinite rank Lie conformal algebras
$\mathfrak{B}(p)$ was introduced in \cite{SXY2018}. In this paper,
we study the structure theory of this class of Lie conformal
algebras. Specifically, we completely determine the conformal
derivations, the conformal biderivations  and certain second
cohomologies of $\mathfrak{B}(p)$. \vspace{5pt}

\noindent{\bf Keywords:} conformal derivation, conformal
biderivation, second cohomology}

\noindent{Mathematics Subject Classification (2000)}: 17B40, 17B56,
17B68, 17B69 }
\parskip .001 truein\baselineskip 6pt \lineskip 6pt
\vspace{10pt}
                     %%第一节，介绍部分

\noindent{\large\bf 1.\ \ Introduction}%
\setcounter{section}{1}\setcounter{equation}{0} \vspace{10pt}

The Lie conformal algebra structure is an axiomatic description of
the operator product expansion of chiral fields in conformal field
theory. On one hand, the theory of Lie conformal algebras provides a
classification of formal distribution Lie algebras. One the other
hand, it is an adequate tool for the study of infinite-dimensional
Lie algebras satisfying the locality property. At present, the
structure theory and representation theory of finite Lie conformal
algebras have been well developed \cite{DK1998,CK1997,CKW1998}. One
of the fundamental works is the classification of finite simple Lie
conformal algebras. It was proved in \cite{DK1998} that a finite
simple Lie conformal algebra is isomorphic to either the Virasoro
conformal algebra, or a current conformal algebra. Besides, the
cohomology theory for Lie conformal algebras was developed in
\cite{BKV1999} and further studied in \cite{DK2009}.

Since the rank of a Lie conformal algebra may not suppose to be
finite, it is natural and necessary to study the infinite rank Lie
conformal algebras. However, the theory of infinite rank Lie
conformal algebras is far from being well developed. The most
important example of simple object is the general Lie conformal
algebra $gc_N$, which plays the same role as the general Lie algebra
$gl_N$ does in the theory of Lie algebras. Partially due to this
importance, $gc_N$ has been deeply studied from different
viewpoints. For example, the finite irreducible conformal modules
over $gc_N$ were classified by Kac, Radul and Wakimoto, see also
\cite{K1997, BM2013}; the finite growth modules over subalgebras of
$gc_N$ containing Virasoro conformal subalgebras were classified in
\cite{BM2013}; certain low dimensional cohomologies of $gc_N$ were
computed in \cite{Su2004}. Recent years, some infinite rank loop
$\ast$-Virasoro type Lie conformal algebras were constructed and
studied, such as the loop Virasoro type \cite{WCY2014},
Heisenberg-Virasoro type \cite{FSW2014} and Schr$\ddot{\rm
o}$dinger-Virasoro type \cite{FSX2016}.

More recently, in \cite{SXY2018}, the authors introduced another
class of infinite rank Lie conformal algebras $\mathfrak{B}(p)$,
whose annihilation algebras have close relation with the Lie
algebras $\mathcal{B}(q)$ of Block type \cite{SXX2012, XYZ2012}.
Particularly, they obtained that a finite irreducible conformal
module over $\mathfrak{B}(p)$ admits a nontrivial extension of a
finite conformal module over Virasoro conformal subalgebra if
$p=-1$. In this paper, we focus on the structure theory of
$\mathfrak{B}(p)$. Specifically, we shall determine the conformal
derivations, the conformal biderivations and certain second
cohomologies of $\mathfrak{B}(p)$, respectively.

This paper is organized as follows. In Section 2, we recall some
definitions on Lie conformal algebra and some properties of
$\mathfrak{B}(p)$. Then, in Sections 3 and 4, we study the conformal
derivations and the conformal biderivations of $\mathfrak{B}(p)$,
respectively. In particular, we show that there exist non-inner
conformal derivations (see Theorem~\ref{theo1}) and non-inner
conformal biderivations (see Theorem~\ref{theo2}) if and only if $p$
is a negative integer. Finally, in Section 5, we compute the second
cohomologies of $\mathfrak{B}(p)$ with coefficients in its trivial
module $\mathbb{C}$ (see Theorem~\ref{theo3}), and certain finite
irreducible conformal modules (see Theorem~\ref{theo4}),
respectively. Again, our results indicate that $\mathfrak{B}(p)$
with a negative parameter $p$ is essentially different from those
with other parameters, and among these $\mathfrak{B}(p)$, the case
with $q=-1$ is the most distinctive one.
%Besides, we find that for any nonzero complex number
%$p$, the second cohomologies of $\mathfrak{B}(p)$ with coefficients
%in its certain finite irreducible conformal modules is trivial.

 Throughout this paper, we denote by
$\mathbb{Z}$, $\mathbb{Z}^{+}$ and $\mathbb{Z}^-$ the sets of
integers, non-negative integers and negative integers, respectively.
Let $\mathbb{C}[\partial]$ be the ring of polynomials in the
indeterminate $\partial$. \vspace{10pt}

\noindent{\large\bf 2.\ \ Preliminaries}%
\setcounter{section}{2}\setcounter{equation}{0} \vspace{10pt}

In this section, we first recall some definitions related to Lie
conformal algebra. Then we recall the Lie conformal algebra
$\mathfrak{B}(p)$ and its some properties. For more details, the
reader can refer to \cite{DK1998,SXY2018} and the references
therein.
\begin{defi} A Lie conformal algebra is a
$\mathbb{C}[\partial]$-module $R$, endowed with a $\lambda$-bracket,
that is a $\mathbb{C}$-linear map $R\otimes
R\rightarrow\mathbb{C}[\lambda]\otimes R$, denoted by $a\otimes
b\mapsto[a_{\lambda}b]$, satisfying the following properties
\begin{eqnarray*}
&\label{C1}[\partial a_{\lambda}b]=-\lambda[a_{\lambda}b],\ \ \
[a_{\lambda}\partial b]=(\partial+\lambda)[a_{\lambda}b]\ \
({\rm conformal\ sesquilinearity}),&\\
&\label{C2}[a_{\lambda}b]=-[b_{-\lambda-\partial}a]\ \ \ ({\rm skew{\rm-}commutativity}),&\\
&\label{C3}[a_{\lambda}[b_{\mu}c]]=[[a_{\lambda}b]_{\lambda+\mu}c]+[b_{\mu}[a_{\lambda}c]]\
\ \ ({\rm Jacobi\ identity}),&
\end{eqnarray*}
for any $a,b,c\in R$. If there exists a finite generating subset
$S\subset R$ such that $S$ generates $R$ as a
$\mathbb{C}[\partial]$-module, then we call that $R$ is a finite
rank Lie conformal algebra. Otherwise, it is called infinite.
\end{defi}

What we mainly consider in this paper is a class of infinite  Lie
conformal algebras
$\mathfrak{B}(p)=\oplus_{i\in\mathbb{Z}^+}\mathbb{C}[\partial]L_i$
with $p$ being a nonzero complex number, satisfying the following
 $\lambda$-brackets:
\begin{eqnarray*}
[{L_i}\,_\lambda L_j]=((i+p)\partial+(i+j+2p)\lambda)L_{i+j}.
\end{eqnarray*}
The Lie conformal algebra $\mathfrak{B}(p)$ contains a Lie conformal
subalgebra
\begin{equation*}
Vir={\rm span}_{\mathbb{C}[\partial]}\{p^{-1}L_0\},
\end{equation*}
which is isomorphic to the Virasoro Lie conformal algebra
\cite{DK1998}. Moreover, if $p$ is a negative integer, then
$\mathfrak{B}(p)$ has another Lie conformal subalgebra
\begin{equation*}
\mathcal{HV}={\rm span}_{\mathbb{C}[\partial]}\{p^{-1}L_0,L_{-p}\},
\end{equation*}
which is isomorphic to Heisenberg-Virasoro conformal algebra
\cite{YW2016}.  In particular, the case $\mathfrak{B}(1)$ is a
maximal subalgebra of the associated graded conformal algebra ${\rm
gr}gc_1$ of the filtered algebra $gc_1$ \cite{SY2011}. Besides, for
any integer $n\geq1$, $\mathfrak{B}(-n)$ contains a series of finite
Lie conformal quotient algebras (see \cite{SXY2018}, Section 2.2).
\begin{defi}
The annihilation algebra $\mathcal{A}(R)$ of a Lie conformal algebra
$R$ is a Lie algebra with $\mathbb{C}$-basis $\{a_{(n)}\,|\,a\in
R,n\in\mathbb{Z}^+\}$ and relations
\begin{equation*}
[a_{(m)},b_{(n)}]=\sum_{k\in\mathbb{\mathbb{Z}^+}}
\left(\!\!\begin{array}{ll}m\\k\end{array}\!\!\right)(a_{(k)}b)_{(m+n-k)},\
\ (\partial a)_{(n)}=-na_{(n-1)},
\end{equation*}
where $a_{(k)}b$ is called the $k$-product, given by $[a_{\lambda}
b]=\sum_{k\in\mathbb{Z}^+}\frac{\lambda^{k}}{k!}a_{(k)}b$.
\end{defi}

Recall from \cite{SXY2018} that the annihilation algebra
$\mathcal{A}(\mathfrak{B}(p))$ of $\mathfrak{B}(p)$ is spanned by
$\{L_{i,m}\,|\,i\in\mathbb{Z}^+,m\in\mathbb{Z}_{\geq-1}\}$ with
relations
\begin{eqnarray*}
[L_{i,m},L_{j,n}]=((j+p)(m+1)-(i+p)(n+1))L_{i+j,m+n}.
\end{eqnarray*}
The Lie algebra $\mathcal{A}(\mathfrak{B}(p))$ is related to the Lie
algebra $\mathcal{B}(q)$ of Block type \cite{SXX2012,SXX2013},
Hence, this Lie conformal algebra $\mathfrak{B}(p)$ is  called a
{\it Lie conformal algebra of Block type}\cite{SXY2018}.

 \vspace{10pt}

\noindent{\large\bf 3.\ \ Conformal derivations of $\mathfrak{B}(p)$
}\setcounter{section}{3}\setcounter{theo}{0}\setcounter{equation}{0}
\vspace{10pt} %(see \cite{DK1998} for details).

Recall \cite{DK1998} that a $\mathbb{C}$-linear map
$D_\lambda:\mathfrak{B}(p)\rightarrow \mathfrak{B}(p)[\lambda]$ is
called a {\it conformal derivation} of $\mathfrak{B}(p)$ if
\begin{eqnarray*}
&D_\lambda(\partial a)=(\partial+\lambda)(D_\lambda a),&\\
&D_{\lambda}[a_\mu b]=[(D_\lambda
a)_{\lambda+\mu}b]+[a_\mu(D_\lambda b)],&
\end{eqnarray*}
for any $a,b\in \mathfrak{B}(p)$. Denote by ${\rm
Der}(\mathfrak{B}(p))$ the space of all conformal derivations. For
any $a\in \mathfrak{B}(p)$, one can define a conformal derivation
$({\rm \bf ad}a)_\lambda$ of $\mathfrak{B}(p)$ by
$$({\rm\bf ad}a)_{\lambda}b=[a_{\lambda}b],\ b\in R.$$
All derivations of this kind are called {\it inner}. Denoted by
${\rm Inn}(\mathfrak{B}(p))$ the space of all inner derivations.

Let $D_{\lambda}\in{\rm Der}(\mathfrak{B}(p))$. Suppose that
\begin{equation*}
D_\lambda(L_0)=\sum_{j\in\mathbb{Z}^+}a_j(\lambda,\partial)L_j
\end{equation*}
for some $a_j(\lambda,\partial)\in\mathbb{C}[\lambda,\partial]$.

Next, we shall discuss the coefficients $a_j(\lambda,\partial)$.

By applying $D_\lambda$ to $[{L_0}_\mu L_0]=p(\partial+2\mu)L_0$ and
considering the coefficients of $L_j$, we can obtain
\begin{eqnarray}{\label{equa1}}
\begin{array}{lll}
&p(\partial+\lambda+2\mu)a_j(\lambda,\partial)\vspace{8pt}\\
=\!\!\!\!&((j+p)\partial+(j+2p)(\lambda+\mu))a_j(\lambda,-\lambda-\mu)
+(p\partial+(j+2p)\mu)a_j(\lambda,\partial+\mu).
\end{array}
\end{eqnarray}
Set $a_j(\lambda,\partial)=\sum_{k=0}^{m}
c_{j,k}(\lambda)\partial^k$ with $c_{j,m}(\lambda)\neq0$, where
$c_{j,k}(\lambda)\in\mathbb{C}[\lambda]$. Assume that $m>1$. Then
comparing the coefficients of $\partial^m$ in (\ref{equa1}), one can
get $(p\lambda-(j+mp)\mu) c_{j,m}(\lambda)=0$, which gives
$c_{j,m}(\lambda)=0$ since $p\neq0$. This contradicts with our
assumption. Thus, one can set
$a_j(\lambda,\partial)=c_{j,0}(\lambda)+c_{j,1}(\lambda)\partial$.
By substituting this into the equation (\ref{equa1}) and then
comparing the coefficients of $\partial$, we can obtain
\begin{equation}\label{equ00}
(j+p)c_{j,0}(\lambda)=(j+2p)\lambda c_{j,1}(\lambda).
\end{equation}

\vspace{6pt} (1) The case $p\notin\mathbb{Z}^-$.\vspace{6pt}

In this case, from (\ref{equ00}), we have
$c_{j,0}(\lambda)=\frac{j+2p}{j+p}\lambda c_{j,1}(\lambda)$. Thus
$a_j(\lambda,\partial)=(\partial+\frac{j+2p}{j+p}\lambda)c_{j,1}(\lambda)$.
Setting $\beta_j=\frac{c_{j,1}(-\partial)}{j+p}L_j$, then we can
check that $({\rm\bf
ad}\beta_j)_\lambda(L_0)=a_j(\lambda,\partial)$. Hence, by taking
$\gamma=\sum_{j\in\mathbb{Z}^+}\beta_j$ and letting
$\bar{D}_\lambda=D_\lambda-({\rm\bf ad}\gamma)_\lambda$, we can
obtain
\begin{equation}{\label{equ01}}
\bar{D}_\lambda(L_0)=0.
\end{equation}
Now, by using $\bar{D}_\lambda(L_0)=0$ and applying
$\bar{D}_\lambda$ to $[L_0\,_\mu L_j]=(p\partial+(j+2p)\mu)L_j$ with
$j>0$, one has
\begin{equation}\label{equ03}
[L_0\,_\mu \bar{D}_\lambda(L_j)]=(p(\partial+\lambda)+(j+2p)\mu)
\bar{D}_{\lambda}(L_j).
\end{equation}
Set
$\bar{D}_\lambda(L_j)=\sum_{k\in\mathbb{Z}^+}e_{j,k}(\lambda,\partial)L_k$
with $e_{j,k}\in\mathbb{C}[\lambda,\partial]$. Then considering the
coefficients of $L_k$ in (\ref{equ03}), one has
\begin{equation*}
(p\partial+(k+2p)\mu)e_{j,k}(\lambda,\partial+\mu)
=(p(\partial+\lambda)+(j+2p)\mu)e_{j,k}(\lambda,\partial).
\end{equation*}
Taking $\mu=0$ in the above equation, we have $p\lambda
e_{j,k}(\lambda,\partial)=0$. Since $p\neq0$, we have
$e_{j,k}(\lambda,\partial)=0$ for any $k\in\mathbb{Z}^+$ and $0\neq
j\in\mathbb{Z}^+$. Thus $\bar{D}_\lambda(L_j)=0$ for $0\neq
j\in\mathbb{Z}^+$. This togethers with (\ref{equ01}) gives
$\bar{D}_\lambda(L_j)=0$ for any $j\in\mathbb{Z}^+$, namely,
\begin{equation}\label{equ07}
D_\lambda(L_j)={(\rm\bf ad}\gamma)_\lambda(L_j),\ \
j\in\mathbb{Z}^+.
\end{equation}
\vspace{6pt}

(2) The case $p\in\mathbb{Z}^-$.\vspace{6pt}

In this case, it follows from (\ref{equ00}) and $p\neq0$ that
$c_{-p,1}(\lambda)=0$. Thus, there exists some
$c(\lambda)\in\mathbb{C}[\lambda]$ such that
$a_{-p}(\lambda,\partial)=c(\lambda)$. Denote by
$\bar{\gamma}=\sum_{-p\neq j\in\mathbb{Z}^+}\beta_j$, where
$\beta_j=\frac{c_{j,1}(-\partial)}{j+p}L_j$. Then, from the same
arguments as in Case (1), we can get
$\bar{D}_\lambda(L_0)=c(\lambda)L_{-p}$, where
$\bar{D}_\lambda=D_\lambda-{(\rm\bf ad}\bar{\gamma})_\lambda$. Let
$g(\lambda)=c(\lambda)-c(0)$, then $g(\lambda)$ divided by
$\lambda$. Thus, from setting
$\tilde{\gamma}=p^{-1}g(-\partial)(-\partial)^{-1}L_{-p}$, we can
get
\begin{equation}\label{equ02}
(\tilde{D}_\lambda(L_0)=cL_{-p},
\end{equation}
where $\tilde{D}_\lambda=\bar{D}_\lambda-{(\rm\bf
ad}\tilde{\gamma})_\lambda$, $c=c(0)$. Now, set
$\tilde{D}_\lambda(L_j)=\sum_{l\in\mathbb{Z}^+}d_{j,l}(\lambda,\partial)L_l$
with $d_{j,l}\in\mathbb{C}[\lambda,\partial]$ and $j>0$. Then, by
applying $\tilde{D}_\lambda$ to $[L_0\,_\mu
L_j]=(p\partial+(j+2p)\mu)L_j$ and comparing the coefficients of
$L_l$, one has
\begin{eqnarray*}
&&(p(\partial+\lambda)+(j+2p)\mu)d_{j,l}(\lambda,\partial)\\
&=\!\!\!&\delta_{l,-p+j}c(j+p)(\lambda+\mu)+(p\partial+(l+2p)\mu)d_{j,l}(\lambda,\partial+\mu).
\end{eqnarray*}
From discussions in Case (1), we get $d_{j,l}=0$ for $l\neq-p+j$.
Now assume that $l=-p+j$. Then, taking $l=-p+j$ and $\mu=0$, we can
get $d_{j,-p+j}(\lambda,\partial)=cp^{-1}(j+p)$. Thus, we obtain
$\tilde{D}_\lambda(L_j)=cp^{-1}(j+p)L_{-p+j}$ with $j>0$. Combining
this with (\ref{equ02}), we get $
\tilde{D}_\lambda(L_j)=cp^{-1}(j+p)L_{-p+j}$ for any
$j\in\mathbb{Z^+}$. Now, Define a $\mathbb{C}$-linear map
$D_{\lambda}^{p}$ as follows:
\begin{equation}\label{out}
D_{\lambda}^{p}(L_j)=(j+p)L_{-p+j},\ \ j\in\mathbb{Z}^+,\ \
p\in\mathbb{Z^-}.
\end{equation}
One can check that $D_{\lambda}^{p}$ is a non-inner conformal
derivation of $\mathfrak{B}(p)$. Thus, we obtain
\begin{equation}\label{equ09}
D_\lambda(L_j)={(\rm\bf
ad}{(\bar{\gamma}+\tilde\gamma)})_\lambda(L_j)+p^{-1}cD_{\lambda}^{p}(L_j),\
\ j\in\mathbb{Z}^+,\ \ p\in\mathbb{Z^-}.
\end{equation}

Hence, from (\ref{equ07}) and (\ref{equ09}), the following theorem
is immediate.
\begin{theo}\label{theo1} Denote by ${\rm
Der}(\mathfrak{B}(p))$ the space of all conformal derivations, ${\rm
Inn}(\mathfrak{B}(p))$ the space of all inner derivations. Let
$D_{\lambda}^{p}$ be as those given in (\ref{out}). We have

\begin{eqnarray*}
{\rm
Der}(\mathfrak{B}(p))=\left\{\begin{array}{lllllllll}
{\rm Inn}(\mathfrak{B}(p))&{\rm if}\ \ p\notin{\mathbb{Z}^-},\vspace{5pt}\\
{\rm
Inn}(\mathfrak{B}(p))\oplus\mathbb{C}D_{\lambda}^{p}&{\rm if}\ \ p\in\mathbb{Z}^-.
\end{array}\right.
\end{eqnarray*}
\end{theo}

\begin{rem}
From Theorem \ref{theo1}, we see that the Heisenberg-Virasoro
conformal algebra $\mathcal{HV}$ has a non-inner conformal
derivation $D_{\lambda}^{p}$ defined as in (\ref{out}).
\end{rem}

 \noindent{\large\bf 4.\ \ Conformal biderivations of  $\mathfrak{B}(p)$
}\setcounter{section}{4}\setcounter{theo}{0}\setcounter{equation}{0}
\vspace{10pt}

In the section, we shall study the conformal biderivations of
$\mathfrak{B}(p)$. First we list some definitions introduced in
\cite{DK2009,ZCY2019}.
\begin{defi}
A $\mathbb{C}$-bilinear map
$\phi_\lambda:\mathfrak{B}(p)\times\mathfrak{B}(p)\rightarrow\mathfrak{B}(p)[\lambda]$
is called a {\it conformal bilinear map} of $\mathfrak{B}(p)$ if
\begin{equation}{\label{R1}} \phi_\lambda(\partial x,y)=-\lambda
\phi_\lambda(x,y),\ \ \phi_{\lambda}(x,\partial
y)=(\partial+\lambda)\phi_\lambda(x,y).
\end{equation}
Furthermore, we call $\phi_{\lambda}$ {\it skew-symmetric} if
\begin{equation}\label{R2}
\phi_\lambda(x,y)=-\phi_{-\partial-\lambda}(y,x).
\end{equation}
\end{defi}
\begin{rem}
Recall Section 2.2 in \cite{DK2009}  that a conformal
$\mathbb{C}$-bilinear is a 2-$\lambda$-bracket on $\mathfrak{B}(p)$
with coefficients in $\mathfrak{B}(p)$, satisfying skew-symmetric.
\end{rem}
\begin{defi}
A conformal $\mathbb{C}$-bilinear map $d_\lambda$ of
$\mathfrak{B}(p)$ is a {\it conformal biderivation} of
$\mathfrak{B}(p)$ if $d_{\lambda}$ is skew-symmetric and satisfies
the relation
\begin{eqnarray*}
d_{\lambda}(x,[y_\mu z])=[(d_\lambda
(x,y))_{\lambda+\mu}z]+[y_\mu(d_\lambda(x,z) )].
\end{eqnarray*}
\end{defi}
One can see that if $d_{\lambda}$ is a conformal biderivation, then
$d_{\lambda}(x,\cdot)$ is a conformal derivation of
$\mathfrak{B}(p)$. From \cite{ZCY2019}, we get that if $d_\lambda$
is a biderivation, then the identity
\begin{equation}\label{bi01}
[(d_\mu(x,y))_{\mu+\gamma}[u_\lambda v]]=[[x_\mu
y]_{\mu+\gamma}d_\lambda(u,v)]
\end{equation}
holds for any $x,y,u,v\in \mathfrak{B}(p)$.

Fix a complex number $c$ and  consider the map
${d^c_\lambda}:R\times R\rightarrow R[\lambda]$ defined by
\begin{equation}\label{bi6}
{d^c_{\lambda}}(x,y)=c[x_\lambda y].
\end{equation}
Then, one can immediately see that $d^c_\lambda$ is a conformal
biderivation. We call the conformal biderivation $d_{\lambda}^{c}$
the  {\it inner conformal biderivation}.

Let $d_\lambda$ be a conformal biderivation of $\mathfrak{B}(p)$,
then one can set
\begin{equation}{\label{N1}}
d_\lambda(L_i,L_j)=\sum_{k\in\mathbb{Z}^+}f_{i,j}^k(\partial,\lambda)L_k
\end{equation}
for some
$f^k_{i,j}(\partial,\lambda)\in\mathbb{C}[\partial,\lambda]$.

In the following parts of this section, we shall describe the
coefficients $f_{i,j}^k(\partial,\lambda)$.

Letting $x=L_i,y=L_j,u=L_m,v=L_n$ in (\ref{bi01}), we have
\begin{eqnarray}\label{bi100}
\begin{array}{ll}
&(\overline{m}(\partial+\mu+\gamma)+(\overline{m}+\overline{n})\lambda)\displaystyle
\sum_{k\in\mathbb{Z}^+} (\overline{k}\partial
+(\overline{m}+\overline{n}+k)(\mu+\gamma))f_{i,j}^k(-\mu-\gamma,\mu)L_{k+m+n}\vspace{8pt}\\
=\!\!\!\!&(\overline{j}\mu-\overline{i}\gamma)\displaystyle\sum_{k\in\mathbb{Z}^+}
(\overline{i+j}\partial+(\overline{i}+\overline{j}+k)(\mu+\gamma))f_{m,n}^k(\partial+\mu+\gamma,\lambda)L_{k+i+j},
\end{array}
\end{eqnarray}
where $\overline{x}=x+p$, $x\in\{i,j,m,n,i+j,k\}$\vspace{6pt}

Letting $m=i,n=j$ and comparing the coefficients of $L_{k+i+j}$, we
have
\begin{eqnarray}\label{bi4}
\begin{array}{ll}
&(\overline{i}(\partial+\mu+\gamma)+(\overline{i}+\overline{j})\lambda)
(\overline{k}\partial+(\overline{i}+\overline{j}+k)(\mu+\gamma))f_{i,j}^k(-\mu-\gamma,\mu)\vspace{8pt}\\
=\!\!\!\!&(\overline{j}\mu-\overline{i}\gamma)
(\overline{i+j}\partial+(\overline{i}+\overline{j}+k)(\mu+\gamma))f_{i,j}^k(\partial+\mu+\gamma,\lambda).
\end{array}
\end{eqnarray}

\begin{lemm}\label{lemma6}
If $p>0$, there exists some $c\in\mathbb{C}$ such that
\begin{equation*}
d_\lambda(L_i,L_j)=d_{\lambda}^c(L_i,L_j)
\end{equation*}
for any $i,j\in\mathbb{Z}^+$.
\end{lemm}
\begin{proof}

By considering the power of $\partial$ in (\ref{bi4}), one can set
$f_{i,j}^k(\partial,\lambda)=c_{i,j}^{k,0}(\lambda)+c_{i,j}^{k,1}(\lambda)\partial$.
Using this in (\ref{bi4}) and comparing the coefficients of
$\partial^2$, we have
\begin{equation}{\label{bi3}}
\overline{i}(c_{i,j}^{k,0}(\mu)-(\mu+\gamma)c_{i,j}^{k,1}(\mu))\overline{k}
=(\overline{j}\mu-\overline{i}\gamma)c_{i,j}^{k,1}(\lambda)\overline{i+j}.
\end{equation}
By comparing the coefficients of $\gamma$ in (\ref{bi3}), we have
\begin{equation}\label{bi19}
\overline{k}c_{i,j}^{k,1}(\mu)=\overline{i+j}c_{i,j}^{k,1}(\lambda).
\end{equation}
Then $c_{i,j}^{k,1}(\mu)=c_{i,j}^{k,1}(\lambda)$ for $k=i+j$, which
forces $c_{i,j}^{k,1}(\mu)\in\mathbb{C}$. Thus we can write
$c_{i,j}^{k,1}$ instead of $c_{i,j}^{k,1}(\mu)$. Using this in
(\ref{bi3}), we have
$c_{i,j}^{k,0}(\mu)=(\overline{i})^{-1}(\overline{i}+\overline{j})
c_{i,j}^{k,1}\mu$. If $k\neq i+j$, then from (\ref{bi19}), we have
$c_{i,j}^{k,1}(\mu)=0$. Together this with (\ref{bi3}), we get
$c_{i,j}^{k,0}(\mu)=0$. Hence, we can set
\begin{equation}{\label{bi10}}
f_{i,j}^k(\partial,\mu)=\delta_{k,i+j}c_{i,j}(\overline{i}\partial+(\overline{i}+\overline{j})\mu)
\end{equation}
for any $k\in\mathbb{Z}^+$,  where
$c_{i,j}=(\bar{i})^{-1}c_{i,j}^{k,1}\in\mathbb{C}$.

Letting $m=n=0$ in (\ref{bi100}) and using (\ref{bi10}), one can get
\begin{eqnarray}{\label{bi00}}
(\partial+\mu+\gamma+2\lambda)
(\overline{j}\mu-\overline{i}\gamma)((\overline{i+j}\partial+(\overline{i}+\overline{j})(\mu+\gamma))(c_{i,j}-c_{0,0})=0.
\end{eqnarray}
Taking $\mu=-\gamma$ in (\ref{bi00}), we have
\begin{equation}\label{bi000}
(\overline{i}+\overline{j})\overline{i+j}(c_{i,j}-c_{0,0})(\partial^2+2\partial\lambda)\gamma=0,
\end{equation}
which gives $c_{i,j}=c_{0,0}$ since $p>0$. Thus, one can set
$f_{i,j}^k(\partial,\mu)=\delta_{k,i+j}c(\overline{i}\partial+(\overline{i}+\overline{j})\mu)
$ for any $i,j\in\mathbb{Z^+}$, where $c=c_{0,0}$. Hence, the lemma
follows from (\ref{N1}).
\end{proof}
\begin{lemm}\label{lemma7}
If $p<0$, then there exists some $\bar{c}\in\mathbb{C}$ such that

\begin{eqnarray*}
d_\lambda(L_i,L_j)=(1-\delta_{i+j,-p})d^{\bar{c}}_{\lambda}(L_i,L_j)
\end{eqnarray*}
for any $i,j\in\mathbb{Z}^+$.
\end{lemm}
\begin{proof}
We consider the identities (\ref{bi4}) from four cases.\vspace{6pt}

(1) $\overline{i}\neq0$, $\overline{i+j}\neq0$. If
$\overline{k}\neq0$, then considering the coefficients of $\partial$
in (\ref{bi4}) and using the same argument as in Lemma \ref{lemma6},
we can obtain the equations (\ref{bi10})--(\ref{bi000}). For
$\overline{i}+\overline{j}\neq0$, one can get $c_{i,j}=c_{0,0}$ from
(\ref{bi000}). For $\overline{i}+\overline{j}=0$, by taking $\mu=0$
in (\ref{bi00}), we have
$\overline{i}(\partial+\gamma+2\lambda)\overline{i+j}(c_{i,j}-c_{0,0})\gamma\partial=0$,
which forces $c_{i,j}=c_{0,0}$. Thus
$d_\lambda(L_i,L_j)=d_{\lambda}^{\bar{c}}(L_i,L_j)$, where
$\bar{c}=c_{0,0}$. If $\overline{k}=0$, then considering the power
of $\partial$ in (\ref{bi4}), one can set
$f_{i,j}^{-p}(\partial,\lambda)=c^{-p}_{i,j}(\lambda)$ for some
$c^{-p}_{i,j}(\lambda)\in\mathbb{C}[\lambda]$. Using this in
(\ref{bi4}) and comparing the coefficients of $\gamma^2$, we have
$\overline{i}(c^{-p}_{i,j}(\mu)+c^{-p}_{i,j}(\lambda))\overline{i+j}=0$,
which forces $c^{-p}_{i,j}(\lambda)=0$ since
$\overline{i},\overline{i+j}\neq0$. This gives
$f_{i,j}^{-p}(\partial,\lambda)=0$. Thus, in this case we have
\begin{equation}\label{N2}
d_\lambda(L_i,L_j)=d_{\lambda}^{\bar{c}}(L_i,L_j),\ i\neq-p,\
i+j\neq-p.
\end{equation}
for some $\bar{c}\in\mathbb{C}$ and any
$i,j\in\mathbb{Z}^+$.\vspace{6pt}

(2) $\overline{i}\neq0$, $\overline{i+j}=0$. In this case, if
$\overline{k}\neq0$, then from (\ref{bi4}), one can set
$f_{i,j}^{k}(\partial,\lambda)=e^{k,0}_{i,j}(\lambda)+e^{k,1}_{i,j}(\lambda)\partial+e^{k,2}_{i,j}(\lambda)\partial^2$.
Using this in (\ref{bi4}), we can get
\begin{eqnarray}\label{bi11}
\begin{array}{lll}
&(\overline{i}\bar{\partial}+p\lambda)\bar{\partial}(e^{k,0}_{i,j}(\mu)+e^{k,1}_{i,j}(\mu)(-\mu-\gamma)+e^{k,2}_{i,j}(\mu)(-\mu-\gamma)^2)\vspace{8pt}\\
=\!\!\!\!&(\overline{j}\mu-\overline{i}\gamma)(\mu+\gamma)
(e^{k,0}_{i,j}(\lambda)+e^{k,1}_{i,j}(\lambda)\bar{\partial}+e^{k,2}_{i,j}(\lambda)\bar{\partial}^2)
\end{array}
\end{eqnarray}
where $\bar{\partial}=\partial+\mu+\gamma$. Then comparing the
coefficients of $\gamma^4$ in the above equation, we can get
$\bar{i}\bar{k}(e^{k,2}_{i,j}(\mu)+e^{k,2}_{i,j}(\lambda)=0$, which
forces $e^{k,2}_{i,j}(\lambda)=0$ since $\bar{i}\bar{k}\neq0$. Using
this  and comparing the coefficients of $\gamma^3$ in (\ref{bi11}),
we can set $e^{k,1}_{i,j}(\mu)=e^{k,1}_{i,j}\in\mathbb{C}$,
substituting into (\ref{bi11}) and comparing the coefficients of
$\partial^2$, one can deduce $e^{k,1}_{i,j}=e^{k,0}_{i,j}(\mu)=0$.
Thus we have
$f_{i,j}^{k}(\partial,\lambda)=\delta_{k,-p}f_{i,j}^{-p}(\partial,\lambda)$.
By using this in (\ref{bi100}) and take $m=n=0$, we can get
$(\partial+\mu+\gamma+2\lambda)
f_{i,j}^{-p}(-\mu-\gamma,\mu)(\mu+\gamma)=0$, which forces
$f_{i,j}^{-p}(\partial,\mu)=0$. Thus,
$f_{i,j}^{k}(\partial,\lambda)=0$ for any $i,j,k\in\mathbb{Z}^+$,
namely,
\begin{equation}\label{N3}
d_{\lambda}(L_i,L_j)=0,\ i\neq-p,\ i+j=-p.
\end{equation}

(3)  $\overline{i}=0$, $\overline{j}\neq0$. In this case, letting
$\mu=-\gamma$ in (\ref{bi4}) and using $\overline{j}\neq0$, one can
obtain
\begin{eqnarray}\label{bi77}
j\mu f_{-p,j}^k(\partial,\lambda)=\overline{k}\lambda
f_{-p,j}^k(0,\mu).
\end{eqnarray}

Assume that $j\neq0$. If $\overline{k}=0$, then
$f_{-p,j}^{-p}(\partial,\lambda)=0$ from (\ref{bi77}). If
$\overline{k}\neq0$, then from (\ref{bi77}), we can safely set
$f_{-p,j}^k(\partial,\lambda)=\delta_{k,-p+j}d_{-p,j}\lambda$ for
some $d_{-p,j}\in\mathbb{C}$. This together with taking $i=m=-p$ in
(\ref{bi100}), then considering the coefficients of $\partial$,
 we have $j((n+p)d_{-p,j}-(j+p)d_{-p,n})\lambda\mu=0$
for any $0\neq n\in\mathbb{Z}^+$, which gives
$d_{-p,j}=\tilde{c}(j+p)$ for some $\tilde{c}\in\mathbb{C}$ since
$j\neq0$. Hence,
\begin{equation}\label{N10}
f_{-p,j}^k(\partial,\lambda)=\delta_{k,-p+j}\tilde{c}(j+p)\lambda.
\end{equation}
Now, by taking $i=-p$, $m=-p+1$ in (\ref{bi100}) and using
(\ref{bi10}) and (\ref{N10}), then comparing the coefficients of
$\partial^2$, we have $j(j+p)(\bar{c}-\tilde{c})\mu=0$ for any
$0\neq j\in\mathbb{Z}^+$, which forces $\tilde{c}=\bar{c}$.  Now
suppose that $j=0$. Since  $d_\lambda(L_{0},L_{-p})=0$ by
(\ref{N3}), we have $d_\lambda(L_{-p},L_{0})=0$. Thus, in this case
we have
\begin{equation}\label{N5}
d_{\lambda}(L_{-p},L_j)=(1-\delta_{-p+j,-p})d^{\bar{c}}_{\lambda}(L_{-p},L_j),
\end{equation}
for any $j\in\mathbb{Z}^+$, where $\bar{c}$ is defined in
(\ref{N2}).

 (4)
$\overline{i}=0$, $\overline{j}=0$. By using
$[{L_{-p}}_{\lambda}{L_{-p}}]=0$ in (\ref{bi01}), we get
$d_{\lambda}(L_{-p},L_{-p})$ is in the center of
$[\mathfrak{B}(p)_{\lambda}\mathfrak{B}(p)]$. Thus we have
\begin{equation}\label{N6}
d_{\lambda}(L_{-p},L_{-p})=0.
\end{equation}
Hence, the lemma holds from (\ref{N2}), (\ref{N3}), (\ref{N5}) and
(\ref{N6}).
\end{proof}

The following theorem follows immediately from Lemmas \ref{lemma6}
and \ref{lemma7}.
\begin{theo}\label{theo2}
Let $d_\lambda$ be a conformal biderivation of $\mathfrak{B}(p)$
with $0\neq p\in\mathbb{C}$. We have

{\rm(1)} If $p\notin{\mathbb{Z}^-}$, then $d_{\lambda}$ is an inner
conformal biderivation.

{\rm(2)} If $p\in\mathbb{Z}^-$, then, for any $i,j\in\mathbb{Z}^+$,
there exists $c\in\mathbb{C}$ such that
\begin{eqnarray*}
d_\lambda(L_i,L_j)=(1-\delta_{i+j,-p})d^c_{\lambda}(L_i,L_j).
\end{eqnarray*}
where $d_{\lambda}^c$ is defined in (\ref{bi6}). Note that
$p^{-1}L_0$ generates the Virasoro conformal algebra. Thus, every
conformal biderivation on the Virasoro Lie conformal algebra is
inner. Besides, if ${ d}_{\lambda}$ is a conformal biderivation of
$\mathcal{HV}$, then
\begin{equation*}
{ d}_{\lambda}(L_0,L_0)= {c}[{L_0}_\lambda L_0],\ {
d}_{\lambda}(L_0,L_{-p})={  d}_{\lambda}(L_{-p},L_{-p})=0.
\end{equation*}

\end{theo}

 \vspace{6pt}
 \noindent{\large\bf 5.\ \ Second Cohomologies of $\mathfrak{B}(p)$
}\setcounter{section}{5}\setcounter{theo}{0}\setcounter{equation}{0}
\vspace{10pt}

 First, let us  recall some  definitions from \cite{DK2009}.

\begin{defi}Let $A$ and $M$ be a $\mathbb{C}[\partial]$-modules, and denote by
$\partial^M$ the action $\partial$ on $M$. A $2$-$\lambda$-bracket
on $A$ with coefficients in $M$ is a $\mathbb{C}$-linear map
\begin{eqnarray*}
c:\ A\otimes A&\rightarrow&
\mathbb{C}[\lambda]\otimes M,\\
a\otimes b&\mapsto&\{{a}_\lambda b\}_{c},
\end{eqnarray*}
satisfying the following conditions:
\begin{eqnarray}
\label{D1}&\{\partial a_{\lambda}b\}_{c}=-\lambda\{a_{\lambda}
b\}_{c}, \ \ \{a_{\lambda}\partial
b\}_{c}=(\lambda+\partial^M)\{a_{\lambda}b\}_{c}\ \ {\rm(sesquilinearity)},&\\
\label{D2}&\{b_{\lambda} a\}_{c}=-\{a_{-\lambda-\partial^M}b\}_{c}\
\ ({\rm skew\!-\!\rm symmetry}).&
\end{eqnarray}
\end{defi}
Denote by $C^2(A,M)$ the space of all $2$-$\lambda$-brackets on $A$
with coefficients in $M$. Recall from \cite{DK2009} that there is a
differential $d$ on the space $C^2(A,M)$, and the equation $dc=0$
can be written as follows:
\begin{eqnarray}\label{D3}
\begin{array}{lll}  &a_{\lambda}\{b_{\mu}z\}_{c}-b_{\mu}\{a_\lambda
z\}_{c}+z_{-\lambda-\mu-\partial^M}\{a_{\lambda}b\}_{c}\vspace{8pt}\\
 &+\,\{a_\lambda[b_\mu z]\}_{c}-\{b_\mu[a_\lambda
z]\}_{c}+\{z_{-\lambda-\mu-\partial^M}[a_\lambda b]\}_{c}=0,
\end{array}
\end{eqnarray}
for every $a,b,z\in A$.

Let $\phi:A\rightarrow M$ be a $\mathbb{C}[\partial]$-module
homomorphism. Then it follows from \cite{DK2009} that exact elements
$c=d\phi$ are of the form
\begin{eqnarray}{\label{dphi}}
\{a_\lambda
b\}_{d\phi}=a_\lambda\phi(b)-b_{-\lambda-\partial^M}\phi(a)-\phi([a_\lambda
b]).
\end{eqnarray}

 Hence, the second cohomology can be defined as
follows.
\begin{eqnarray}
H^2(A,M)=\{c:A\otimes A\rightarrow \mathbb{C}[\lambda]\otimes
M|(\ref{D1}){\rm-}(\ref{D3})\ {\rm hold}\}/\{c\ {\rm of\ the\ form\
}(\ref{dphi})\}.
\end{eqnarray}
We say that the elements of $H^2(A,M)$ are {\it $2$-cocycles}, and
call $c\in H^2(A,M)$ is {\it trivial} if there exists some $d\phi$
defined in (\ref{dphi}) such that $c=d\phi$.
\begin{defi}A module over a Lie conformal algebra $A$ is a
$\mathbb{C}[\partial]$-module $M$, endowed with a $\lambda$-{\it
action}, that is a $\mathbb{C}$-linear map $A\otimes
M\rightarrow\mathbb{C}[\lambda]\otimes M$, denoted by $a\otimes
m\mapsto a_\lambda m$, such that for any $a,b\in A$, $m\in M$,
\begin{eqnarray*}
&(\partial a)_\lambda m=-\lambda a_\lambda m,\ \ a_\lambda(\partial
m)=(\partial+\lambda)a_\lambda m,&\\
&a_\lambda(b_\mu m)-b_\mu(a_\lambda m)=[a_\lambda
b]_{\lambda+\mu}m.&
\end{eqnarray*}
\end{defi}

Let $M$ be a nontrivial free conformal module of rank one over
$\mathfrak{B}(p)$. One can see that the one-dimensional vector space
$\mathbb{C}$ can be regarded as a module (called a trivial module)
over $\mathfrak{B}(p)$ with the action of $\partial$,$L_i$ being
zero. Recall from \cite{SXY2018} that there are some
$\Delta,\alpha,\beta \in\mathbb{C}$ such that

(1) If $p\neq-1$, then $M\cong M_{\Delta,\alpha}$, where
\begin{eqnarray*}
M_{\Delta,\alpha}=\mathbb{C}[\partial]v,\ \ {L_0}_\lambda
v=p(\partial+\Delta\lambda+\alpha)v,\ \ {L_i}_\lambda v=0,\ i\geq1.
\end{eqnarray*}

(2) If $p=-1$, then $M\cong M_{\Delta,\alpha,\beta}$, where
\begin{eqnarray*}
M_{\Delta,\alpha,\beta}=\mathbb{C}[\partial]w,\ \ {L_0}_\lambda
w=-(\partial+\Delta\lambda+\alpha)w,\ \ {L_1}_\lambda w=\beta w,\ \
{L_i}_\lambda w=0,\ i\geq2.
\end{eqnarray*}
Furthermore, the module $M_{\Delta,\alpha}$ (resp.,
$M_{\Delta,\alpha,\beta}$) is irreducible if and only if
$\Delta\neq0$ (resp., $\Delta\neq0$ or $\beta\neq0$).

In the following sections, we shall study the second cohomologies
$H^2(\mathfrak{B}(p),\mathbb{C})$, $H^2(\mathfrak{B}(p),M_{\Delta,
{\alpha}})$ with $p\neq-1$ and $H^2(\mathfrak{B}(-1),M_{\Delta,
{\alpha},\beta})$ for $\alpha\neq0$, respectively.\vspace{6pt}

\noindent{{\it 5.1.\ \ Second Cohomology of $\mathfrak{B}(p)$ with
trivial coefficients}}\vspace{6pt}

In this section, we shall compute the second cohomology
$H^2(\mathfrak{B}(p),\mathbb{C})$, where $\partial$, $L_i$
$(i\in\mathbb{Z}^+)$ act by zero on $\mathbb{C}$.

Assume that $c\in C^2(\mathfrak{B}(p),\mathbb{C})$ be a
$2$-$\lambda$-bracket. Since $\partial$ acts by zero on
$\mathbb{C}$, we can write (\ref{D3}) and (\ref{dphi}) as follows:

\begin{eqnarray}
\label{DD3}&\{a_\lambda[b_\mu z]\}_{c}-\{b_\mu[a_\lambda
z]\}_{c}+\{z\,_{-\lambda-\mu}[a_\lambda b]\}_{c}=0,&\\
\label{dd4}&\{a_\lambda b\}_{d\phi}=-\phi([a_\lambda b]).&
\end{eqnarray}
Then, replacing $a,b,c$ by $L_i,L_j,L_k$ in (\ref{DD3}),
respectively, we have
\begin{eqnarray*}
\{{L_i}\,_\lambda[{L_j}\,_\mu L_k]\}_{c}-\{{L_j}\,_\mu[{L_i}\,_\lambda
L_k]\}_{c}+\{{L_k}\,_{-\lambda-\mu}[{L_i}\,_\lambda L_{j}]\}_{c}=0,
\end{eqnarray*}
which gives
\begin{eqnarray}{\label{Main1}}
\begin{array}{lll}
&((j+p)\lambda+(j+k+2p)\mu)\{{L_{i}}\,_\lambda{L_{j+k}}\}_{c}
-((i+p)\mu+(i+k+2p)\lambda)\{{L_{j}}\,_\mu{L_{i+k}}\}_{c}\vspace{8pt}\\
&=((j+p)\lambda-(i+p)\mu)\{{L_{i+j}}\,_{\lambda+\mu}{L_{k}}\}_{c}.
\end{array}
\end{eqnarray}

\begin{lemm}{\label{lemma1}}
If $p>0$, then replacing $c$ with $c+d\phi$, we can get
\begin{equation*}
\{{L_j}\,_{\lambda}L_k\}_{c}=\delta_{j+k,0}a\lambda^{3},\
j,k\in\mathbb{Z}^+,
\end{equation*}
where $a\in\mathbb{C}$ and $d\phi$ is as in (\ref{dd4}).
\end{lemm}
\begin{proof}
Define a $\mathbb{C}[\partial]$-module homomorphism
$\phi:\mathfrak{B}(p) \rightarrow\mathbb{C}$ by the formula
\begin{eqnarray*}
\phi(L_i)=(i+2p)^{-1}\frac{d}{d\lambda}\{{L_0}\,_{\lambda}L_i\}_{c}|_{\lambda=0},\
i\in\mathbb{Z}^+.
\end{eqnarray*}
Since
$\{{L_{0}}\,_{\lambda}L_{i}\}_{d\phi}=-\lambda(\frac{d}{d\,_{\lambda}}\{{L_{0}}\,_{\lambda}L_{i}\}_{c}|_{\lambda=0})$
by (\ref{dd4}), replacing $c$ with  $c+d\phi$, we have
\begin{equation}\label{main2}
\frac{d}{d\lambda}\{{L_0}\,_{\lambda}L_i\}_{\bar{c}}|_{\lambda=0}=0,\
i\in\mathbb{Z}^+.
\end{equation}

Letting $i=0$ and taking the derivative  of both sides of
(\ref{Main1}) with respect to $\lambda$, we get
\begin{eqnarray}{\label{main3}}
\begin{array}{lll}
&(j+p)\{{L_{0}}\,_\lambda{L_{j+k}}\}_{{c}}
+((j+p)\lambda+(j+k+2p)\mu)\frac{d}{d\lambda}\{{L_{0}}\,_\lambda{L_{j+k}}\}_{ {c}}-(k+2p)\{{L_{j}}\,_\mu{L_{k}}\}_{ {c}}\vspace{8pt}\\
&=(j+p)\{{L_{j}}\,_{\lambda+\mu}{L_{k}}\}_{
{c}}+((j+p)\lambda-p\mu)\frac{\partial}{\partial\lambda}\{{L_{j}}\,_{\lambda+\mu}{L_{k}}\}_{
{c}}.
\end{array}
\end{eqnarray}
Note that $\{{{L_0}}\,_{0}L_k\}_{{d\phi}}=0$. Thus, by letting
$i=j=\lambda=0$ in (\ref{Main1}) and using $p>0$, one can deduce
$\{{L_0}\,_{0}L_k\}_c=0$ for all $k\in\mathbb{Z}^+$. Thus, taking
$\lambda=0$ in (\ref{main3}) and using (\ref{main2}), we have
\begin{eqnarray}\label{jia4}
(j+k+3p)\{{L_{j}}\,_\mu{L_{k}}\}_{{c}}
=p\mu\frac{\partial}{\partial\lambda}\{{L_{j}}\,_{\lambda+\mu}{L_{k}}\}_{{c}}|_{\lambda=0},\
j,k\in\mathbb{Z}^+.
\end{eqnarray}
Set
$\{{L_{j}}\,_\mu{L_{k}}\}_{{c}}=\sum_{l=0}^{m}h_l(j,k)\mu^l\in\mathbb{C[\mu]}$.
Assume that $m\geq3$. Then from the above equation, we can get
\begin{equation}{\label{jia3}}
(j+k+3p)h_0(j,k)+\sum_{l=1}^m(j+k+(3-l)p)h_l(j,k)\mu^l=0,\
j,k\in\mathbb{Z}^+,
\end{equation}
which gives
\begin{eqnarray}
\label{jia6}&h_0(j,k)=h_1(j,k)=h_2(j,k)=0,&\\
\label{jia7}&(j+k+(3-l)p)h_l(j,k)=0,\ \ 3\leq l\leq m.&
\end{eqnarray}
Fix $j$ and $k$. If $j+k=0$, then (\ref{jia6}) and (\ref{jia7})
imply that $\{{L_0}\,_{\lambda}L_0\}_c=h_{3}(0,0)\mu^{3}$. Now, let
$j+k\neq0$. Assume that $(l_0-3)p=j+k$ for some $4\leq l_0\leq m$.
Then, it follows from (\ref{jia6}) and (\ref{jia7}) that
$\{{L_j}\,_{\lambda}L_k\}_c=\delta_{j+k,(l_0-3)p}h_{l_0}(j,k)\mu^{l_0}$.
Then, letting $i=0$ and $j+k=(l_0-3)p$ in (\ref{Main1}), we have
\begin{eqnarray*}
\begin{array}{ll}
&((j+p)\lambda+(l_0-1)p\mu)\{{L_0}\,_{\lambda}L_{(l_0-3)p}\}_{{c}}
-(p\mu+(k+2p)\lambda)h_{l_0}(j,k)\mu^{l_0}\vspace{8pt}\\
&=((j+p)\lambda-p\mu)h_{l_0}(j,k)(\lambda+\mu)^{l_0}.
\end{array}
\end{eqnarray*}
By considering the coefficients of $\lambda^2\mu^{l_0-1}$ and
$\lambda^3\mu^{l_0-2}$ in the above equation, we get
$(2j+(3-l_0)p)h_{l_0}(j,k)=0$ and $(3j+(5-l_0)p)h_{l_0}(j,k)=0$,
respectively, which forces $h_{l_0}(j,k)=0$ for any $j+k>0$. Hence,
the proof follows by  setting $a=h_3(0,0)$.

\end{proof}

Now we assume that $p<0$.

Letting $k=0$ in (\ref{Main1}), one has
\begin{eqnarray}{\label{gg2}}
\begin{array}{lll}
&((j+p)\lambda+(j+2p)\mu)\{{L_{i}}\,_\lambda{L_{j}}\}_{c}
-((i+p)\mu+(i+2p)\lambda)\{{L_{j}}\,_\mu{L_{i}}\}_{c}\vspace{8pt}\\
&=((i+p)\mu-(j+p)\lambda)\{{L_{0}}\,_{-\lambda-\mu}{L_{i+j}}\}_{c}.
\end{array}
\end{eqnarray}
On the other hand, by taking $i=j=0$ in (\ref{Main1}), we have
\begin{eqnarray}{\label{main5}}
(p\lambda+(k+2p)\mu)\{{L_{0}}\,_\lambda{L_{k}}\}_{c}
-(p\mu+(k+2p)\lambda)\{{L_{0}}\,_\mu{L_{k}}\}_{c}=p(\lambda-\mu)\{{L_{0}}\,_{\lambda+\mu}{L_{k}}\}_{c}.
\end{eqnarray}
Set
$\{{L_{0}}\,_\lambda{L_{k}}\}_{c}=\sum_{i=0}^{n}g_i(k)\lambda^i\in\mathbb{C}[\lambda]$
with $g_n(k)\neq0$. If $n>3$, then considering the coefficients of
$\lambda^{n}$ in (\ref{main5}), we have
$\big((n-3)p-k\big)g_n(k)=0$, which forces $g_n(k)=0$ since $n>3$.
Thus, we can set
$\{{L_{0}}\,_\lambda{L_{k}}\}_{c}=g_0(k)+g_1(k)\lambda+g_2(k)\lambda^2+g_3(k)\lambda^3.$
By substituting this into (\ref{main5}) and then considering the
power of $\lambda$, one can deduce
\begin{equation}\label{gg}
\{{L_{0}}\,_\lambda{L_{k}}\}_{c}=\delta_{k,-2p}g_0(k)+g_1(k)\lambda
+\delta_{k,-p}g_2(k)\lambda^2+\delta_{k,0}g_3(k)\lambda^3.
\end{equation}
%=================================================================================
\begin{lemm}\label{lemma2}

Let $p<0$. For any $i,j\in\mathbb{Z}^+$ and some
$a_1,a_2,a_3,b\in\mathbb{C}$, we have

\begin{eqnarray*}
&\{L_i{\,_{\lambda}}L_j\}_c=
\delta_{i+j,0}a_1\lambda^3+\delta_{i+j,-2p}(a_2(j+p)+a_3\lambda),\
j\neq1\ \ {\rm if}\ \ p=-1,&\\
&\{{{L_1}\,_{\lambda}}L_1\}_c=b\lambda\ \ {\rm if}\ \ p=-1.&
\end{eqnarray*}

\end{lemm}

\begin{proof}
We consider three cases.

 \vspace{6pt} (1) The case
$p\notin\mathbb{Z}^{-}$, $2p\notin\mathbb{Z}^-$.\vspace{6pt}

From (\ref{gg}) that
$\{{L_{0}}\,_\lambda{L_{k}}\}_{c}=g_1(k)\lambda+\delta_{k,0}g_3(k)\lambda^3$.
Now, define a $\mathbb{C}[\partial]$-linear map
$\phi:\mathfrak{B}(p)\rightarrow \mathbb{C}$,
\begin{equation*}
\phi(L_k)=(k+2p)^{-1}g_1(k),
\end{equation*}
for any $k\in\mathbb{Z}^+$. Then replacing $c$ by $c+d\phi$, one can
set
\begin{equation}\label{gg3}
 \{{L_{0}}\,_\lambda{L_{k}}\}_c=\delta_{k,0}g_3(0)\lambda^3,\
 k\in\mathbb{Z}^+.
\end{equation}

Assume that $i+j\neq0$. Then using (\ref{gg3}) in (\ref{gg2}), we
have
\begin{eqnarray}{\label{eq1}}
((j+p)\lambda+(j+2p)\mu)\{{L_{i}}\,_\lambda{L_{j}}\}_{c}
+((i+p)\mu+(i+2p)\lambda)\{{L_{i}}\,_{-\mu}{L_{j}}\}_{c}=0.
\end{eqnarray}
Since $p,2p\notin\mathbb{Z}^-$, by considering the power of
$\lambda$ in the above equation, we can set
$\{{L_{i}}\,_\lambda{L_{j}}\}_{c}=f(i,j)\in\mathbb{C}$. Thus, it
follows from (\ref{eq1}) that $(i+j+3p)f(i,j)=0$, which gives
\begin{equation}\label{2}
 \{L_i{\,_{\lambda}}L_j\}_c=\delta_{i+j,-3p}f(i,j),\ \ i+j\neq0.
\end{equation}
Suppose that $3p\in\mathbb{Z}^-$. Set $\tilde{p}=-3p-1$. Then by
taking $i=1$ and $k=\tilde{p}-j\geq0$ in (\ref{Main1}), we can get
\begin{equation*}
(j+p)(f(j+1,\tilde{p}-j)-f(j,-3p-j)-f(1,\tilde{p}))=0,\ 0\leq
j\leq\tilde{p},
\end{equation*}
Note that $p\notin\mathbb{Z}^-$. Thus from the above identities that
$f(j+1,\tilde{p}-j)=(j+1)f(1,\tilde{p})$, which implies
$f(1,\tilde{p})=-(3p)^{-1}f(-3p,0)$. Thus, $f(1,\tilde{p})=0$ since
$f(-3p,0)=-\{{L_0}\,_{-\lambda}L_{-3p}\}_c=0$ by (\ref{gg3}). Hence,
$f(i,j)=f((i-1)+1,\tilde{p}-(i-1))=if(1,\tilde{p})=0$ for any
$i+j=-3p$ with $i\neq0$. Combing this with (\ref{gg3}) and
(\ref{2}), we have
$\{{L_{i}}\,_\lambda{L_{j}}\}_{c}=\delta_{i+j,0}a_1\lambda^3$, where
$a_1=g_3(0)$.\vspace{6pt}

(2) The case $p\notin\mathbb{Z}^{-}$,
$2p\in\mathbb{Z}^-$.\vspace{6pt}

It follows from (\ref{gg}) that
\begin{equation*}
\{{L_{0}}\,_\lambda{L_{k}}\}_{c}=\delta_{k,-2p}g_0(k)+g_1(k)\lambda+\delta_{k,0}g_3(k)\lambda^3.
\end{equation*}
Define
a $\mathbb{C}[\partial]$-linear map $\phi:\mathfrak{B}(p)\rightarrow
\mathbb{C}$ as follows:
\begin{equation}{\label{gg6}}
\phi(L_k)=(k+2p)^{-1}g_1(k)\ for\ k\neq-2p,\ \phi(L_{-2p})=0.
\end{equation}
Then, replacing $c$ by $c+d\phi$, one can set
\begin{eqnarray}{\label{gg1}}
\begin{array}{ccc}
&\{{L_{0}}\,_\lambda{L_{k}}\}_c=0,\ k\neq0,-2p,\
\{{L_{0}}\,_\lambda{L_{0}}\}_c=g_3(0)\lambda^3,\
\vspace{8pt}\\
&\{{L_{0}}\,_\lambda{L_{-2p}}\}_c=g_0(-2p)+g_1(-2p)\lambda.
\end{array}
\end{eqnarray}

If $i+j\neq-2p,0$, then from (\ref{gg2}), we can get (\ref{eq1}).
Thus, if $i\neq-2p$, then from (\ref{eq1}), we also get
$\{L_i{\,_{\lambda}}L_j\}_c\in\mathbb{C}$ since
$p\notin\mathbb{Z}^-$. Using this in (\ref{eq1}), we have
$(i+j+3p)\{L_i{\,_{\lambda}}L_j\}_c=0$. Since $p\notin\mathbb{Z}^-$
but $2p\in\mathbb{Z}^-$, we have $3p\notin\mathbb{Z}^-$. Thus, in
case $i\neq-2p$, we get $\{L_i{\,_{\lambda}}L_j\}_c=0$. For $i=-2p$,
by taking $i=-2p$ in (\ref{eq1}), we immediately get
$\{L_{-2p}{\,_{\lambda}}L_j\}_c=0$. Thus, we obtain
$\{L_i{\,_{\lambda}}L_j\}_c=0$ for $i+j\neq-2p$.

If $i+j=-2p$, then from (\ref{gg2}) and  (\ref{gg1}), we can get
\begin{eqnarray}{\label{jia9}}
\begin{array}{lll}
&((j+p)\lambda+(j+2p)\mu)\{{L_{-2p-j}}\,_\lambda{L_{j}}\}_{c}
-((j+p)\mu+j\lambda)\{{L_{-2p-j}}\,_{-\mu}{L_{j}}\}_{c}\vspace{6pt}\\
&=-(j+p)(\mu+\lambda)(g_0(-2p)-g_1(-2p)(\lambda+\mu)),
\end{array}
\end{eqnarray}
Note that $p\notin\mathbb{Z}^-$. Thus, by considering the power of
$\lambda$ in above equation, one can easily deduce $
\{{L_{-2p-j}}\,_\lambda{L_{j}}\}_{c}=a_2(j+p)+a_3\lambda$, where
$a_2=-p^{-1}g_0(-2p)$, $a_3=g_1(-2p)$.

\vspace{6pt} (3) The case $p\in\mathbb{Z}^-$.\vspace{6pt}

Recall the result in (\ref{gg}),
\begin{equation}\label{Rrm2}
\{{L_{0}}\,_\lambda{L_{k}}\}_{c}=\delta_{k,-2p}g_0(k)+g_1(k)\lambda
+\delta_{k,-p}g_2(k)\lambda^2+\delta_{k,0}g_3(k)\lambda^3.
\end{equation}
Define $\phi$ as those given in (\ref{gg6}). Then, replacing $c$ by
$c+d\phi$, one can set
\begin{eqnarray}
\nonumber&\{{L_{0}}\,_\lambda{L_{k}}\}_c=0,\ k\neq0,-p,-2p,\ \
\{{L_{0}}\,_\lambda{L_{-2p}}\}_c=g_0(-2p)+g_1(-2p)\lambda,
&\\
\label{Rem3a}&\{{L_{0}}\,_\lambda{L_{-p}}\}_c=g_2(-p)\lambda^2,\ \
\{{L_{0}}\,_\lambda{L_{0}}\}_c=g_3(0)\lambda^3.&
\end{eqnarray}

(i) Case $i+j=-2p$. If $j\neq-p$, then from (\ref{jia9}), we can get
$\{{L_{-2p-j}}\,_\lambda{L_{j}}\}_{c}=a_2(j+p)+a_3\lambda$. If
$j=-p$, then $i=-p$. Thus, taking $j=-p$ in (\ref{jia9}), one can
easily deduce $\{L_{-p}{\,_{\lambda}}L_{-p}\}_c=b\lambda$ for some
$b\in\mathbb{C}$. Now, setting $i=-p$ and $k=-p-1$ in (\ref{Main1}),
then comparing the coefficients of $\lambda^2$, we can get
$(p+1)(b-a_3)=0$, which gives $b=a_3$ for $p\neq-1$. Thus, for any
$j\in\mathbb{Z}^+$, we get
\begin{eqnarray}
\label{Rem3b}&\{{L_{-2p-j}}\,_\lambda{L_{j}}\}_{c}=a_2(j+p)+a_3\lambda,\
j\neq1\ \ {\rm if}\ \ p=-1,&\\
\nonumber&\{{L_{1}}\,_\lambda{L_{1}}\}_{c}=b\lambda\ \ {\rm if}\ \
p=-1.&
\end{eqnarray}

(ii) Case $i+j=-p$. In this case, from (\ref{gg2}) that
\begin{eqnarray}{\label{gg10}}
\begin{array}{lll}
&((j+p)\lambda+(j+2p)\mu)\{{L_{-p-j}}\,_\lambda{L_{j}}\}_{c}
-(j\mu+(j-p)\lambda)\{{L_{-p-j}}_{-\mu}{L_{j}}\}_{c}\vspace{8pt}\\
&=-g_2(-p)(j\mu+(j+p)\lambda)(\lambda+\mu)^2.
\end{array}
\end{eqnarray}
By considering the power of $\lambda$ in the above equation, we can
set
$\{{L_{-p-j}}\,_\lambda{L_{j}}\}_{c}=h_0(j)+h_1(j)\lambda+h_2(j)\lambda^2$.
Using this in (\ref{gg10}), we can deduce
$h_0(j)=h_2(j)=h_2(j)=g_2(-p)=0$. Thus,
$\{{L_{i}}\,_\lambda{L_{j}}\}_{c}=0$ for any $i+j=-p$.

(iii) Case $i+j=0$. It follows from (\ref{Rrm2}) that
$\{{L_{0}}\,_\lambda{L_{0}}\}_{c}=a_1\lambda^3$, where $a_1=g_3(0)$.

(iv) Case $i+j\neq-2p,-p,0$. In this case, we also have (\ref{eq1}).
Thus, If $j\neq-p$ and $i\neq-2p$, then we can get (\ref{2}), which
gives $\{L_i{\,_{\lambda}}L_j\}_c=0$. One the other hand, using
(\ref{eq1}), one can easily deduce $\{L_i{\,_{\lambda}}L_j\}_c=0$
for $j=-p$ and $i\neq-2p$, or $j\neq-p$ and $i=-2p$, respectively.
Besides, by taking $j=-p$ and $i=-2p$ in (\ref{gg2}), we can get
$\{L_{-2p}{\,_{\lambda}}L_{-p}\}_c\in\mathbb{C}$. Combing this with
setting $j=-p$, $k=-2p-i$, $i=1$ and $2$ in (\ref{Main1}),
respectively, one can check that
$(1+p)\{L_{-2p}{\,_{\lambda}}L_{-p}\}_c
=(2+p)\{L_{-2p}{\,_{\lambda}}L_{-p}\}_c=0$, which forces
$\{L_{-2p}{\,_{\lambda}}L_{-p}\}_c=0$.  This completes the proof.

\end{proof}

 The following  statements follow from straightforward
verifications.

(1) For any $0\neq p\in\mathbb{C}$, the $2$-$\lambda$-bracket
$\alpha$ defined by
\begin{equation}\label{alpha}
\{{L_{i}}\,_\lambda L_{i}\}_{\alpha}=\delta_{i+j,0}\lambda^3
\end{equation}
is a  nontrivial 2-cocycle.

(2) If $2p\in\mathbb{Z}^-$, then the following
$2$-$\lambda$-brackets $\beta$, $\tilde{\beta}$ and $\bar{\beta}$
defined by (all other terms are vanishing)
\begin{eqnarray}
\label{beta1}&\{{L_i}_{\,\lambda}L_j\}_{{\beta}}=\delta_{i+j,-2p}(j+p),\
\ \{{L_1}{\,_\lambda}L_1\}_{\tilde{\beta}}=\lambda\ \
{\rm if}\ \ p=-1,&\\
\label{beta2}&\{{L_i}{\,_\lambda}L_j\}_{\bar{\beta}}=\delta_{i+j,-2p}\lambda,\
j\neq1\ \ {\rm if}\ \ p=-1.
\end{eqnarray}
are three nontrivial 2-cocycles.

\begin{theo}\label{theo3}
Let $\alpha$, $\beta$, $\bar{\beta}$ and $\tilde{\beta}$ be as in
(\ref{alpha})--(\ref{beta2}). We have
\begin{eqnarray*}\label{Main67}
H^2(\mathfrak{B}(p),\mathbb{C})=\left\{\begin{array}{lllllllll}
\mathbb{C}\alpha&{\rm if}\ \ 2p\notin\mathbb{Z}^-,\vspace{5pt}\\
\mathbb{C}\alpha\oplus\mathbb{C} \beta
\oplus\mathbb{C} \bar{\beta}&{\rm if}\ \ 2p\in\mathbb{Z}^-,\ p\neq-1,\vspace{5pt}\\
\mathbb{C}\alpha\oplus\mathbb{C}\beta
\oplus\mathbb{C}\bar{\beta}\oplus\mathbb{C}\tilde{\beta}&{\rm if}\ \ p=-1.
\end{array}\right.
\end{eqnarray*}
\end{theo}
\begin{proof}

Let $c$ be an element of $H^2(\mathfrak{B}(p),\mathbb{C})$. Then
from the Lemmas \ref{lemma1} and {\ref{lemma2}} that there exist
some $a,a_1,a_2,a_3,b\in\mathbb{C}$ such that

(1) If $2p\notin\mathbb{Z}^-$, then
$\{{L_j}_{\,\lambda}L_k\}_{c}=\delta_{j+k,0}a\lambda^{3}$ for any
$i,j\in\mathbb{Z}^+$, namely, $c=a\alpha$.

(2) If $2p\in\mathbb{Z^-}$ with $p\neq-1$, then
$\{L_i{_{\,\lambda}}L_j\}_c=\delta_{i+j,0}a_1\lambda^3+\delta_{i+j,-2p}(a_2(j+p)+a_3\lambda)
$ for any $i,j\in\mathbb{Z}^+$, which imply
$c=a_1\alpha+a_2\beta+a_3\bar{\beta}$.

(3) If $p=-1$, then (all other terms are vanishing)
\begin{equation*}
\{L_0{\,_{\lambda}}L_0\}_c=a_1\mu^3,\
\{L_0{\,_{\lambda}}L_2\}_c=-\{L_2{\,_{-\lambda}}L_0\}_c=a_2+a_3\lambda,\
\{L_1{\,_{\lambda}}L_1\}_c=b\lambda,
\end{equation*}
which give $c=a_1\alpha+a_2\beta+a_3\bar{\beta}+b\tilde{\beta}$.
This completes the proof.
\end{proof}

 \noindent{{\it 5.2.\ \ Second cohomology of $\mathfrak{B}(p)$ $(resp., \mathfrak{B}(-1))$
with coefficients in $M_{\Delta,{\alpha}}(resp.,
M_{\Delta,{\alpha},\beta})$} \vspace{6pt}

In this section, we only consider the case $\alpha\neq0$.

Assume that $p\neq-1$. Take
 $c\in C^2(\mathfrak{B}(p),M_{\Delta, {\alpha}})$ and suppose that
 $\{{L_i}\,_\lambda{L_j}\}_c=f_{i,j}(\partial,\lambda)v\in
 \mathbb{C}[\lambda]\otimes M_{\Delta, {\alpha}}$. Then  by taking $a=L_0$, $\lambda=0$,
 $b=L_j$ and $z=L_k$ with $jk\neq0$ in (\ref{D3}), respectively, we have

\begin{eqnarray}{\label{E1}}
p {\alpha}
f_{j,k}(\partial,\mu)=-((j+p)\partial+(j+k+2p)\mu)f_{0,j+k}(\partial,0),\
\ j,k>0.
\end{eqnarray}
Letting $a=b=z=L_0$ and $\lambda=0$ in (\ref{D3}), one has
\begin{equation}\label{E2}
\begin{array}{lll}
 {\alpha}f_{0,0}(\partial,\mu)&\!\!=\!\!\!\!&(\partial+\Delta\mu+{\alpha})f_{0,0}(\partial+\mu,0)\vspace{8pt}\\
&\!\!\!\!&-(\partial+\Delta(-\mu-\partial^M)+ {\alpha})f_{0,0}(-\mu,0)-(\partial+2\mu)f_{0,0}(\partial,0))
\end{array}
\end{equation}
Letting $a=b=L_0$, $z=L_k$ with $k\neq0$ and $\lambda=0$ in (\ref{D3}), one can obtain
\begin{eqnarray}{\label{E3}}
p {\alpha}f_{0,k}(\partial,\mu)=p(\partial+\Delta\mu+ {\alpha})f_{0,k}(\partial+\mu,0)
-(p\partial+(k+2p)\mu)f_{0,k}(\partial,0)
\end{eqnarray}
Note that $f_{k,0}(\partial,\mu)=-f_{0,k}(\partial,-\mu-\partial)$.
Now, define a $\mathbb{C}[\partial]$-linear map $\phi$ from
$\mathfrak{B}(p)$ to $M_{\Delta, {\alpha}}=\mathbb{C}[\partial]v$
as $ \phi(L_i)={\alpha  p}^{-1}f_{0,i}(\partial,0)v$.
Then replacing $c$ by $c-d\phi$ and using (\ref{E1})--(\ref{E3}),
one can obtain that $\{{{L_i}\,_\lambda}L_{j}\}_c=0$ for any
$i,j\in\mathbb{Z}^+$. Thus, we get $H^2(\mathfrak{B}({p}),M_{\Delta, {\alpha}})=0$.

Next, we study the second cohomology of $\mathfrak{B}(-1)$ with
coefficients in $M_{\Delta, {\alpha},\beta}$.

Let  $\bar{c}\in C^2(\mathfrak{B}(-1),M_{\Delta, {\alpha},\beta})$ and
 set $\{{L_{i\,_\lambda}}{L_j}\}_{\bar c}=g_{i,j}(\partial,\lambda)v\in
 \mathbb{C}[\lambda]\otimes M_{\Delta, {\alpha},\beta}$.
 Then  taking $a=b=L_0$, $z=L_1$ and $\lambda=0$ in (\ref{D3}), we have

\begin{eqnarray}{\label{D4}}
 {\alpha}
g_{0,1}(\partial,\mu)=(\partial+\Delta\mu+ {\alpha})g_{0,1}(\partial+\mu,0)-g_{0,0}(-\mu,0)\beta+(\partial+\mu)g_{0,1}(\partial,0).
\end{eqnarray}
Letting $a=L_1$, $b=L_0$, $z=L_k$ with $k\neq0,1$ and $\mu=0$ in
(\ref{D3}), we have

\begin{eqnarray}{\label{D5}}
 {\alpha} g_{1,k}(\partial,\lambda)=g_{0,k}(\partial+\lambda,0)\beta-(k-1)\lambda
g_{0,k+1}(\partial,0),\ \ k\neq0,1.
\end{eqnarray}
Using $g_{1,0}(-\lambda,\lambda)=-g_{0,1}(-\lambda,0)$, taking $a=z=L_1$, $b=L_0$ and $\mu=0$ in (\ref{D3}), we have

\begin{eqnarray}{\label{D6}}
 {\alpha}
g_{1,1}(\partial,\lambda)=\beta(g_{0,1}(\partial+\lambda,0)-g_{0,1}(-\lambda,0)).
\end{eqnarray}
Define a $\mathbb{C}[\partial]$-linear map $\varphi$ from
$\mathfrak{B}(-1)$ to
$M_{\Delta {\alpha},\beta}=\mathbb{C}[\partial]v$ as
$\varphi(L_i)= {\alpha}^{-1}g_{0,i}(\partial,0)v$. Note that $g_{i,j}(\partial,\lambda)$ also satisfy
the identities (\ref{E1})--(\ref{E3}) for $i,j>1$. These together with (\ref{D4})--(\ref{D6}) imply that
$\{{{L_i}\,_\lambda}L_{j}\}_{\bar c}=\{{{L_i}\,_\lambda}L_{j}\}_{d\varphi}$ for any $i,j\in\mathbb{Z}^+$.
Hence, we have $H^2(\mathfrak{B}(-1),M_{\Delta, {\alpha},\beta})=0$.

The following theorem is immediate.
\begin{theo}\label{theo4}
There is no non-trivial second cohomology of $\mathfrak{B}(p)(p\neq-1)(resp., \mathfrak{B}(-1))$
with coefficients in $M_{\Delta,{\alpha}}(resp.,
M_{\Delta,{\alpha},\beta})$ for $\alpha\neq0$.
\end{theo}

\end{document}